\documentclass[reqno]{amsart}

\usepackage{amsrefs}
\usepackage{hyperref}

\newtheorem{theorem}{Theorem}
\newtheorem{lemma}[theorem]{Lemma}
\newtheorem{proposition}[theorem]{Proposition}

\newtheorem{problem}[theorem]{Problem}

\theoremstyle{exercise}

\theoremstyle{definition}
\newtheorem{definition}[theorem]{Definition}
\newtheorem{notation}[theorem]{Notation}
\newtheorem{conjecture}[theorem]{Conjecture}

\usepackage{graphicx}
\usepackage{pstricks, enumerate}

\numberwithin{equation}{section}
\numberwithin{theorem}{section}

\newenvironment{proofof}[1]{\noindent{\bf Proof of #1.}\hspace*{1em}}{\qed\bigskip}

\newcommand{\intav}[1]{\mathchoice {\mathop{\vrule width 6pt height 3 pt depth  -2.5pt
\kern -8pt \intop}\nolimits_{\kern -6pt#1}} {\mathop{\vrule width
5pt height 3  pt depth -2.6pt \kern -6pt \intop}\nolimits_{#1}}
{\mathop{\vrule width 5pt height 3 pt depth -2.6pt \kern -6pt
\intop}\nolimits_{#1}} {\mathop{\vrule width 5pt height 3 pt depth
-2.6pt \kern -6pt \intop}\nolimits_{#1}}}

\newcommand{\intavl}[1]{\mathchoice {\mathop{\vrule width 6pt height 3 pt depth  -2.5pt
\kern -8pt \intop}\limits_{\kern -6pt#1}} {\mathop{\vrule width 5pt
height 3  pt depth -2.6pt \kern -6pt \intop}\nolimits_{#1}}
{\mathop{\vrule width 5pt height 3 pt depth -2.6pt \kern -6pt
\intop}\nolimits_{#1}} {\mathop{\vrule width 5pt height 3 pt depth
-2.6pt \kern -6pt \intop}\nolimits_{#1}}}

\newcommand{\R}{\mathbb{R}}

\newcommand{\Z}{\mathbb{Z}}

\newcommand{\eps}{\epsilon}

\newcommand{\EE}[2]{{\mathbb{E}}\left[{#1}|{#2}\right]}
\newcommand{\CR}[1]{{\rm CR}\left({#1}\right)}
\newcommand{\dist}[2]{{\rm dist}\left({#1},{#2}\right)}

\begin{document}

\title[A Class of Multi-particle  Reinforced  Interacting Random Walks]{A Class of Multi-particle  Reinforced \\  Interacting Random Walks}

\author[Jun Chen]{Jun Chen}
\address{Weizmann Institute of Science, Faculty of Mathematics and Computer Science, POB 26, 76100, Rehovot, Israel.}
\email{jun.chen@weizmann.ac.il}




\date{\today}


\begin{abstract}
We consider a class of multi-particle reinforced interacting random walks. In this model, there are some (finite or infinite)
particles performing random walks on a given (finite or infinite) connected graph, so that each particle has higher probability to visit
neighboring vertices or edges which have been seldom visited by the other particles. Specifically we investigate two particles' vertex-reinforced
interacting random walks on finite complete graphs. By a dynamical approach we prove that the two particles' occupation measure asymptotically has
small joint support almost surely if the reinforcement is strong.
\end{abstract}

\maketitle

\section{Introduction}
``Reinforced random walk" (RRW) is a class of non-Markovian random walk, which has been extensively studied in the last twenty years,  see the survey ~\cite{pemantle2007survey}. RRW is defined on the vertices of an undirected graph in such a way that the probability of a transition from one vertex to another depends on the number of previous transitions along the neighboring edges (respectively, vertices). We refer to such models as ``edge-reinforced random walk" (ERRW) (respectively, ``vertex-reinforced random walk" (VRRW)).\\

RRWs can be defined on all kinds of graphs, either finite or infinite. It also can be defined  by very different reinforcement mechanisms through the weight functions, which could lead to quite different phenomena to occur. A question of interest is the recurrence or transience of RRW. For example, Angel et al. ~\cite{ANGEL2012LOCALIZATION} and Sabot et al. ~\cite{sabot2012recurrent} independently showed that linearly ERRW on any graph with bounded degrees is recurrent. Another phenomenon of interest is localization. For example, Pemantle and  Volkov ~\cite{pemantle1999vertex} showed that VRRW on $\mathbb{Z}$ has finite range, and then Tarr\'{e}s ~\cite{tarres2004vertex} showed that VRRW on $\mathbb{Z}$ eventually gets stuck on five points. Phase transitions are also of interest. For instance, Pemantle ~\cite{pemantle1988phase} showed that ERRW on infinite binary tree can vary from transient to recurrent, depending on the value of an adjustable parameter  measuring the strength of the feedback (reinforcement), and Volkov ~\cite{volkov2006phase} showed that VRRW on $\mathbb{Z}$ with weight function $k^\alpha$ will just visit  $2, 5, \infty$ sites after a large time $T_0$ when $\alpha>1,=1$ or $<1$ respectively. \\

So far, almost all the considered RRW models are one particle's self-interacting with positive (attractive) feedback (reinforcement), i.e.\ the edges (or vertices) already being visited more are more likely to be visited in the future. In analyzing such models, four main methods are commonly used: exchangeability ~\cite{kovchegov2008multi,pemantle2007survey}, branching process embedding ~\cite{davis1990reinforced,limic2003attracting,LIMIC2007ATTRACTING, sellke1994reinforced}, stochastic approximation via martingale methods ~\cite{HILL1980STRONG, robbins1951stochastic}, and dynamical system approach~\cite{benaim1996dynamical, benaim1999dynamics}.  \\

One direction to generalize RRW is to consider multi-particle RRW. Kovchegov ~\cite{kovchegov2008multi} showed for the two particles' edge reinforced process on $\mathbb{Z}$, the two particles meet infinitely often a.s.. In his model, each particle's transition probability is positively reinforced (determined) by both particles' visits on the edges. Except for his paper, the generalization of RRW to multi-particle RRW models with more general reinforced interacting mechanism hasn't appeared yet. In 2010, Itai Benjamini proposed a class of new multi-particle reinforced interacting random walks, and finally our paper is an exploration of that model. \\

\section{The model and statement of result}  \label{definition_of_the_model}

We will define a class of two particles' vertex-reinforced interacting random walks on a connected graph. Denote the two particles by $X$ and $Y$,
and the graph by $G=(V,E)$. At each step, both $X$ and $Y$ will jump to some vertices on $V$. Let $X_k, Y_k$ be $X,Y$'s locations at time $k$ on $V$,
and $N(X,v,n),N(Y,v,n)$ be the number of $X,Y$'s visits to vertex $v$ by time $n$. We also assume that $N(X,v,0)=N(Y,v,0)=1$ for any $v\in V$.
Denote the natural filtration generated by $\{X_k,0\le k \le n\}$ and $\{Y_k,0\le k \le n\}$ by $\mathcal{F}_n(n\in \mathbb{N})$.
Then the stochastic process $(X_n, Y_n)$'s transition probability is defined
\begin{equation} \label{Xmove}
\mathbb{P}(X_{n+1}=v|\mathcal{F}_n)=\frac{1_{v\sim X_n}\cdot w_{N (Y,v,n)}}{\sum_{p\sim X_n} w_{N (Y,p,n)}}
\end{equation}
and
\begin{equation}  \label{Ymove}
\mathbb{P}(Y_{n+1}=v|\mathcal{F}_n)=\frac{1_{v\sim Y_n}\cdot w_{ N (X,v,n)}}{\sum_{p\sim Y_n} w_{N (X,p,n)}},
\end{equation}
where ``$\sim$" represents some vertex relation on $G$ (e.g. nearest-neighbor), and $w_k(k\in \Z_+)$ is a fixed sequence of positive numbers,
 referred to as ``weights".  One natural weight sequence to work with is the nonlinear one $w_k=k^{-\alpha}$ for some $\alpha>0$.

On a finite connected graph $G=(V,E)$ with $V=\{1, \ldots, d\}$, one can set
\begin{equation} \label{randomprocess}
x_i(n)=\frac{N(X,i,n)}{n+d},\ y_i(n)=\frac{N(Y,i,n)}{n+d},\ \forall i\in V
\end{equation}
as $X$ and $Y$'s \emph{empirical occupation measure} on  $V$ by time $n$. Then if $w_k$ (like $k^{-\alpha}$) is homogenous in $k$,
we can rewrite (\ref{Xmove}) and (\ref{Ymove}) as
\begin{equation} \label{measureXmove}
\mathbb{P}(X_{n+1}=i|\mathcal{F}_n)=\frac{1_{i\sim X_n}\cdot \varpi(y_i(n))}{\sum_{j\sim X_n} \varpi(y_j(n))}
\end{equation}
and
\begin{equation}  \label{measureYmove}
\mathbb{P}(Y_{n+1}=i|\mathcal{F}_n)=\frac{1_{i\sim Y_n}\cdot \varpi( x_i(n))}{\sum_{j\sim Y_n} \varpi(x_j(n))},
\end{equation}
where $\varpi(x)$ is a function $\varpi: [0,1] \to \R_{\ge 0}$. Hence, $X,Y$'s transition probabilities are defined by their occupation measure.

In this paper, we will work with the model defined by (\ref{measureXmove}) and (\ref{measureYmove}) on finite complete graphs with
$\varpi(x)=\left[\delta 1_{x\le \delta}+x 1_{x>\delta}\right]^{-\alpha}$ for some fixed $\delta, \alpha>0$. Here we assume $\delta$ is small enough
and the vertex relation ``$\sim$" is that the graph distance between two vertices is less than or equal to 1, i.e. $X,Y$ are lazy random walks.
Then for any $i,j\in V$, (\ref{measureXmove}) and (\ref{measureYmove}) become
\begin{equation} \label{completemeasureXmove}
\mathbb{P}(X_{n+1}=i|\mathcal{F}_n)=\frac{ \left[\delta 1_{y_i(n)\le \delta}+y_i(n) 1_{y_i(n)>\delta}\right]^{-\alpha}}{\sum_{k=1}^d \left[\delta 1_{y_k(n)\le \delta}+y_k(n) 1_{y_k(n)>\delta}\right]^{-\alpha}}
\end{equation}
and
\begin{equation}  \label{completemeasureYmove}
\mathbb{P}(Y_{n+1}=j|\mathcal{F}_n)=\frac{ \left[\delta 1_{x_j(n)\le \delta}+x_j(n) 1_{x_j(n)>\delta}\right]^{-\alpha}}{\sum_{k=1}^d \left[\delta 1_{x_k(n)\le \delta}+x_k(n) 1_{x_k(n)>\delta}\right]^{-\alpha}}.
\end{equation}
Denote $x(n)=(x_1(n),\ldots,x_d(n))$ and $y(n)=(y_1(n), \ldots, y_d(n))$. Set a $2d$ dimensional vector
\begin{equation} \label{main random process}
z(n)=(x(n),y(n))=(x_1(n),\ldots,x_d(n), y_1(n),\ldots,y_d(n)).
\end{equation}
Notice that $z(n)$ is a Markov chain living in $\R^{2d}$. Then we will show $z(n)$'s asymptotic behavior.

Here we want to mention that when $\min_{i,j\in V} \{x_i(n),y_j(n)\}>\delta$, (\ref{completemeasureXmove}) and (\ref{completemeasureYmove})
 are equivalent to the following formulas
\begin{equation} \label{XXmove}
\mathbb{P}(X_{n+1}=i|\mathcal{F}_n)=\frac{N (Y,i,n)^{-\alpha}}{\sum_{k=1}^d N (Y,k,n)^{-\alpha}}, \quad \forall i\in V
\end{equation}
\begin{equation}  \label{YYmove}
\mathbb{P}(Y_{n+1}=j|\mathcal{F}_n)=\frac{N (X,j,n)^{-\alpha}}{\sum_{k=1}^d N (X,k,n)^{-\alpha}}, \quad \forall j \in V,
\end{equation}
which can be easily seen as a class of two particles' \emph{repelling} interacting random walks with nonlinear reinforcement.

Throughout the paper, we use the following notation.

\begin{notation}  \label{notation_collection}
\begin{enumerate}
\item  Let $\Delta$ be the closed $(d-1)-$dimensional  simplex
\[\Delta=\left\{ u\in \R^d:u_i\ge 0, \ \sum_{i\in V} u_i=1 \right\}.
\]
Denote the relative interior of $\Delta$ by $\overset{\circ}{\Delta}$.
\item Let $D$ be the product of two simplices $\Delta \times \Delta$
\[D=\left\{ (u,v)\in \R^{2d}:u_i\ge 0,\ \sum_{i\in V} u_i=1 \ \text{and} \ v_i\ge 0, \ \sum_{i\in V} v_i=1 \right\}.
\]
Denote the relative interior of $D$ by $\overset{\circ}{D}$, and the boundary of $D$ by $\partial D$ ;
\item Let $TD$ be the set identified with the tangent space to $D$ at each point
\[TD=T(\Delta\times \Delta)=\left\{(u,v)\in \mathbb{R}^{2d},\ \sum_{i\in V}u_i=0, \ \sum_{i\in V}v_i=0 \right\}.
\]
\item Let $U$ be the $d-$dimensional vector $(1/d, \ldots, 1/d)$. We also call $U$ the uniform distribution.
\item Let $\Vert \cdot \Vert$ denote the $L^1$ norm on $\mathbb{R}^{2d}$  defined by $\Vert v \Vert=\sum_{i=1}^{2d} |v_i|$.
\end{enumerate}
\end{notation}

Then we have the following theorem for any small positive $\delta$ appearing in the definition of $\varpi(x)$.

\begin{theorem} \label{main theorem}
For any fixed $d\ge 3\in \mathbb{N}$, there exists some $\alpha(d)$ independent of $\delta$, s.t. when $\alpha\ge \alpha(d)$,
 the two components $x(n)$ and $y(n)$ of $z(n)$ in (\ref{main random process}) asymptotically have small joint support bounded by $4\delta$ almost surely,
  i.e.
\[\mathbb{P} \left\{ \exists n_0, \bigcap_{n\ge n_0}\left\{\sum_{i=1}^d x_i(n)y_i(n)<4\delta\right\} \right\}=1.
\]
\end{theorem}




The organization of the rest of this paper is as follows:
In Section \ref{stochastic approximation algorithms}, we will show that $z(n)$ belongs to a class of \emph{stochastic approximation algorithms}. In
Section \ref{the dynamical approach}, we will introduce the \emph{dynamical approach} and conclude that the limit set of $z(n)$ is contained in the \emph{chain recurrent set}
of a semiflow induced by an ordinary differential equation (ODE). In Section \ref{proof of main theorem}, we will prove Theorem
\ref{main theorem}. In Section \ref{further problems}, some further open problems are proposed.
\section{Stochastic approximation algorithms} \label{stochastic approximation algorithms}
In general, a {\it stochastic approximation algorithm} is a discrete time stochastic process whose form can be written as
\begin{equation} \label{stochastic approximation}
z(n+1)-z(n)=\gamma_n H(z(n),\xi(n))
\end{equation}
where $H:\R^m\times\R^k\to\R^m$ is a measurable function that characterizes the algorithm, $\{z(n)\}_{n\ge 0}\subset\R^m$
is the sequence of parameters to be recursively updated, $\{\xi(n)\}_{n\ge 0}\subset\R^k$ is a sequence of random
variables defined on some probability space $(\Omega, \mathcal{G}, \mathbb{P})$, and $\{\gamma_n\}_{n\ge 0}$ is a sequence of ``small"
nonnegative numbers. Such processes were first introduced in the early 50s in the works of Robbins and Monro~\cite{robbins1951stochastic}
and Kiefer and Wolfowitz~\cite{kiefer1952stochastic}.

To show that $z(n)$ in (\ref{main random process})  is a stochastic approximation algorithm, we need to show $z(n)$
satisfies a difference equation of the form (\ref{stochastic approximation}). Observe that  from (\ref{randomprocess})
\begin{eqnarray} \label{differenceeqn}
x_i(n+1)-x_i(n) &=& \frac{N(X,i,n)+1_{X_{n+1}=i}}{n+1+d}-\frac{N(X,i,n)}{n+d} \nonumber \\
            &=& \frac{-x_i(n)+1_{X_{n+1}=i}}{n+1+d}.
\end{eqnarray}
Similarly, we can derive a difference equation for $y_i(n)$. So $z(n)$ satisfies (\ref{stochastic approximation}) with
\begin{equation} \label{definition xin}
\gamma_n= \frac{1}{n+1+d}, \ \xi(n)=(1_{X_{n+1}=1}, \ldots, 1_{X_{n+1}=d}, 1_{Y_{n+1}=1}, \ldots, 1_{Y_{n+1}=d})
\end{equation}
and $H: \R^{2d} \times \R^{2d} \rightarrow \R^{2d}$ defined by
\[
H(z(n),\xi(n))=-z(n)+\xi(n).
\]
That is,
\begin{equation} \label{our stochastic approximation}
z(n+1)-z(n)=\frac{1}{n+1+d}\left( -z(n)+\xi(n) \right).
\end{equation}

To analyze the asymptotic behavior of $z(n)$ in (\ref{our stochastic approximation}), it is convenient to introduce a related ODE as shown in the next section.
\section{The dynamical approach}  \label{the dynamical approach}

The dynamical approach is a method used to analyze stochastic approximations, introduced by Ljung \cite{ljung1977analysis}
and Kushner and Clark \cite{kushner1978stochastic}. The idea is to decouple the stochastic approximation algorithm into its mean part and the other so-called ``noise" part, and then study the asymptotic behavior of the algorithm in terms of  the mean component's behavior.
 This method has been widely studied and inspired many works, such as the book by Kushner and Clark \cite{kushner1978stochastic}, numerous articles by Kushner, and more recently
the book by Benveniste, Metivier, and Priouret \cite{benveniste1990adaptive}.

In the above perspective, our stochastic approximation algorithm can be written as
\begin{align*}
z(n+1)-z(n)=\gamma_n\left\{(-z(n)+\EE{\xi(n)}{\mathcal F_n})+(\xi(n)-\EE{\xi(n)}{\mathcal F_n})\right\},
\end{align*}
where $\mathcal F_n$ was defined in Section \ref{definition_of_the_model}.
Set a map $\pi=(\pi_1, \ldots, \pi_d): \Delta\to \Delta$ with
\begin{equation} \label{stationary}
\pi_i(x)=\frac{\varpi(x_i)}{\sum_{k=1}^d \varpi(x_k)}, \quad \forall x\in \Delta,
\end{equation}
where $\varpi(x)=\left[\delta 1_{x\le \delta}+x 1_{x>\delta}\right]^{-\alpha}$.
Observe that, by (\ref{completemeasureXmove}), (\ref{completemeasureYmove}) and (\ref{definition xin}),
\begin{align*}
\EE{\xi(n)}{\mathcal F_n}=(\pi(y(n)), \pi(x(n))).
\end{align*}
Thus, defining $\{u_n\}_{n\ge 0}\subset\R^{2d}$ by
\begin{align}\label{definition u_n}
u_n=\xi(n)-\EE{\xi(n)}{\mathcal F_n}
\end{align}
and $F=(F_1,\ldots,F_{2d})$ to be a vector field in $\Delta\times \Delta$ with
\begin{equation}\label{definition F, part  1}
F_i(x_1,\ldots,x_d,y_1,\ldots,y_d)=
\begin{cases}
-x_i+\pi_i(y_1,\ldots,y_d), &\text{if $1\le i \le d$; }  \\
-y_{i-d}+\pi_{i-d}(x_1,\ldots,x_d), & \text{if  $d+1\le i \le 2d$, }
\end{cases}
\end{equation}
our random process takes the form
\begin{align}\label{definition system}
z(n+1)-z(n)=\gamma_n\left[F(z(n))+u_n\right].
\end{align}

The above expression is a particular case of a class of stochastic approximation algorithms studied by
Bena{\"\i}m  in \cite{benaim1996dynamical}, on which he related the behavior of the algorithm to a weak notion
of recurrence for the ODE: that of {\it chain-recurrence}. His theorem asserts that, under the assumptions
of Kushner and Clark lemma \cite{kushner1978stochastic}, the accumulation points of $\{z(n)\}_{n\ge 0}$
are contained in the chain-recurrent set of the semiflow generated by the ODE.

In the remaining of this section, we introduce the necessary definitions for semiflows,
then state Bena{\"\i}m's theorem, and conclude the section by proving that our model satisfies
the required conditions of this theorem.

\subsection{Preliminaries on semiflows}

Let $\Gamma\subset \R^m$ be a metric space and  ``$\rm dist(\cdot,\cdot)$" denote the metric. Let $\Phi:\R_{\ge 0}\times \Gamma\rightarrow \Gamma$ be a continuous map.
For simplicity, denote $\Phi(t,x)$ by $\Phi_t(x)$.

\begin{definition}[Semiflow]
A {\it semiflow} on $\Gamma$ is a continuous map $\Phi:\R_{\ge 0}\times \Gamma\rightarrow \Gamma$ such that
\begin{enumerate}[(i)]
\item $\Phi_0$ is the identity on $\Gamma$, and
\item $\Phi_{t+s}=\Phi_t\circ\Phi_s$ for any $t,s\ge 0$.
\end{enumerate}
\end{definition}

In particular, for every continuous vector field $F:\R^m\rightarrow\R^m$ with unique integral curves, we can associate a semiflow on $\R^m$
by the equation
\begin{align*}
\frac{d}{dt}\Phi_t(x)=F(\Phi_t(x))\, ,\ \ \ \forall\,x\in\R^m,\forall\,t\in\R_{\ge 0}.
\end{align*}
If $F$ is Lipschitz, it has unique integral curves.

Fix a semiflow $\Phi$ on $\Gamma\subset\R^m$.

\begin{definition}[Invariant set]
A set $A \subset \Gamma$ is called {\it invariant} if $\Phi_t(A)\subset A$ for every $t\ge 0$.
\end{definition}

\begin{definition}[Equilibrium point]
A point $x\in \Gamma$ is called an {\it equilibrium} if $\Phi_t(x)=x$ for all $t\ge 0$. The {\it equilibrium set} of
$\Phi$ is the set of all equilibrium points.
\end{definition}

When $\Phi$ is induced by a vector field $F$, the equilibrium set coincides with the set
on which $F$ vanishes.

\begin{definition}[Chain-recurrent point]
Given $\rho,T>0$, a point $x\in \Gamma$ is called {\it $(\rho,T)$-recurrent} if there are points
$x_0=x,x_1,\ldots,x_{k-1},x_k=x\in \Gamma$ and real numbers $t_0,t_1,\ldots,t_{k-1}\ge T$ such that
\begin{align*}
\dist{\Phi_{t_i}(x_i)}{x_{i+1}}<\rho, \ \ \ i=0,\ldots,k-1.
\end{align*}
$x$ is said to be {\it chain-recurrent} if it is $(\rho,T)$-recurrent for any $\rho,T>0$.
\end{definition}

We denote by $\CR{\Phi}$ the set of chain-recurrent points. $\CR{\Phi}$ is closed and invariant.

We denote the limit set of a discrete sequence $\{x(n)\}_{n\ge 0}\subset \Gamma$ by $L\left(\{x(n)\}_{n\ge0}\right)$.
The sets describing the asymptotic behavior of the orbits of $\Phi$ are the omega limit sets.
\begin{definition}[Omega limit set]
The omega limit set of $w\in \Gamma$, denoted by $\omega(w)$, is the set of $x\in \Gamma$ such that $\lim_{k\to \infty}\Phi_{t_k} (w)= x$ for some sequence $t_k > 0$ with $\lim_{k\to \infty}t_k=\infty$.
\end{definition}
If $\Gamma$ is compact, $\omega(w)$ is a nonempty, compact, connected and invariant set.

\begin{definition}[Lyapunov function]  \label{definition Lyapunov function}
A continuous map
$L: \Gamma \to\R$ is said to be a {\it Lyapunov function} for some subset $\Lambda \subset \Gamma$ if the
function $t\in \R_{\ge 0} \to L(\Phi_t(x))$ is strictly decreasing along any non-constant orbit $\Phi_t(x) \subset \Lambda$.
\end{definition}

\subsection{A limit set theorem}

The reason we can characterize the limit set of the random process via the chain-recurrent set
of the deterministic semiflow is due to Theorem 1.2 of \cite{benaim1996dynamical} which, to
our purposes, is stated as

\begin{theorem}\label{theorem benaim}
Let $F:\R^m\to\R^m$ be a continuous vector field with unique integral curves,
and let $\{z(n)\}_{n\ge 0}$ be a solution to the recursion
\begin{align*}
z(n+1)-z(n)=\gamma_n\left[F(z(n))+u_n\right],
\end{align*}
where $\{\gamma_n\}_{n\ge 0}$ is a decreasing gain sequence\footnote{$\lim_{n\to\infty}\gamma_n=0$ and
$\sum_{n\ge 0}\gamma_n=\infty$.} and $\{u_n\}_{n\ge 0}\subset \R^m$. Assume that
\begin{enumerate}[(i)]
\item $\{z(n)\}_{n\ge 0}$ is bounded, and
\item for each $T>0$,
\begin{align*}
\lim_{n\to\infty} \sup_k \left\{\left\Vert\sum_{i=n}^{k-1}\gamma_i u_i\right\Vert: \sum_{i=n}^{k-1}\gamma_i \le T \right\}=0.
\end{align*}
\end{enumerate}
Then $L(\{z(n)\}_{n\ge 0})$ is a connected set chain-recurrent for the semiflow induced by $F$.
\end{theorem}

\subsection{The random process (\ref{definition system}) satisfies Theorem \ref{theorem benaim}}

First, note that $\delta 1_{x\le \delta}+x 1_{x>\delta}$ is bounded by $\delta$ and $1$, and $\varpi(x)$ is Lipschitz. Then $\pi$ in (\ref{stationary}) and $F$ in (\ref{definition F, part  1}) are Lipschitz. Meanwhile, $\gamma_n=1/(n+1+d)$ satisfies
\begin{align*}
\lim_{n\to\infty}\gamma_n=0\ \ \text{ and}\ \ \sum_{n\ge 0}\gamma_n=\infty.
\end{align*}
It remains to check condition (ii). For that, let $M_n=\sum_{i=0}^n\gamma_i u_i$. Observe that $\{M_n\}_{n\ge 0}$ is a martingale adapted to
 $\{\mathcal F_{n+1}\}_{n\ge 0}$
\begin{align*}
\EE{M_{n+1}}{\mathcal F_{n+1}}=\sum_{i=0}^n\gamma_i u_i+\EE{\gamma_{n+1}u_{n+1}}{\mathcal F_{n+1}}=\sum_{i=0}^n\gamma_i u_i=M_n.
\end{align*}
Furthermore, because for any $n\ge 0$
\begin{align*}
\sum_{i=0}^n\EE{\Vert M_{i+1}-M_i\Vert^2}{\mathcal F_{i+1}}\le (2d)^2 \cdot\sum_{i=0}^n\gamma_{i+1}^2\le (2d)^2 \cdot \sum_{i\ge 0}\gamma_i^2<\infty\ \text{ a.s.},
\end{align*}
the sequence $\{M_n\}_{n\ge 0}$ converges  to a finite random variable in $\R^{2d}$ almost surely
(see e.g. Theorem 5.4.9 of~\cite{durrett2010probability}). In particular, it is a Cauchy sequence and so condition (ii) holds almost surely.

Now, in view of Theorem \ref{theorem benaim}, we will investigate the chain-recurrent set of semiflow
generated by the ODE
\begin{equation}  \label{DerODE}
\begin{cases}\dfrac{du_i(t)}{dt}=-u_i(t)+\displaystyle \frac{f(v_i)^{-\alpha}}{\sum_{k=1}^d f(v_k)^{-\alpha}}, & i=1,\ldots, d \\
\dfrac{dv_i(t)}{dt}=-v_i(t)+\displaystyle \frac{f(u_i)^{-\alpha}}{\sum_{k=1}^d f(u_k)^{-\alpha}}, & i=1,\ldots, d \\
\end{cases}
\end{equation}
where
\begin{equation}  \label{truncation_function_f}
f(x)=\delta 1_{x\le \delta}+x 1_{x>\delta}, \, i.e. \, f(x)=\varpi(x)^{-\frac{1}{\alpha}}.
\end{equation}
We can rewrite (\ref{DerODE}) in vector form
\begin{equation}  \label{DerODEvector}
\begin{cases}\dfrac{du(t)}{dt}=-u(t)+\displaystyle \pi(v(t))  \\
\dfrac{dv(t)}{dt}=-v(t)+\displaystyle \pi(u(t))  \\
\end{cases}
\quad \text{or} \quad \frac{d\Xi(t)}{dt}=F(\Xi(t))
\end{equation}
where $\Xi(t)=(u(t),v(t))\in D$.

Before moving to the proof of Theorem \ref{main theorem}, we will prove a simple fact regarding (\ref{DerODE}).

\begin{proposition} \label{proposition10}
The domain $D$  is invariant under $\Phi$, the semiflow induced by (\ref{DerODE}).
\end{proposition}

\begin{proof}
Suppose $(u,v)\in \partial D$. Without loss of generality, we can assume that there exists some $i\in V$ such that $u_i=0$. Then by
(\ref{DerODE}), we have
\[\left.\frac{du_i(t)}{dt}\right|_{(u,v)}\ge \inf_{v\in \Delta }\frac{ f(v_i)^{-\alpha}}{\sum_{j=1}^d f(v_j)^{-\alpha}}> 0.
\]
Hence, $F(u,v)$ points inward whenever $(u,v)$ belongs to the boundary of $D$. Thus any forward trajectory based in $D$ remains in $D$.
\end{proof}

\section{Proof of Theorem \ref{main theorem} }  \label{proof of main theorem}
According to Theorem \ref{theorem benaim}, the limit set of  $\{z(n)\}_{n\ge 0}$ is contained in the chain recurrent set,
and so the first step to prove Theorem \ref{main theorem} is to characterize  chain-recurrent set for our specific semiflow induced by (\ref{DerODE}). Recall $U$'s definition in Notation \ref{notation_collection}. We will conclude the proof of Theorem \ref{main theorem} by showing that $\{z(n)\}_{n\ge 0}$ has probability 0 to
converge to the isolated unstable equilibrium $(U,U)$.

\subsection{Chain recurrent set}   \label{chain recurrent set}

\subsubsection{Lyapunov function}   \label{Lyapunov function}

We characterize our chain-recurrent set $\CR{\Phi}$ by introducing a Lyapunov function
\begin{equation} \label{Lyapunov function L(u,v)}
L(u,v)=\sum_{i=1}^d u_i v_i, \quad (u,v)\in D.
\end{equation}
Let $\Phi_t(x)=(u_1(t),\ldots, u_d(t),v_1(t), \ldots, v_d(t))$ ($t\ge 0$) be an orbit of $\Phi$ where $x=(u_1(0),\ldots,\\ u_d(0),v_1(0), \ldots, v_d(0))$. Then
\begin{eqnarray}   \label{derivativeformula}
&& \frac{d}{dt}(L(\Phi_t(x)))  \nonumber \\
&=&\sum_{i=1}^d  v_i(t)\frac{du_i(t)}{dt}+ \sum_{i=1}^d u_i(t) \frac{dv_i(t)}{dt}     \nonumber     \\
&=&\sum_{i=1}^d  v_i\left(-u_i+\frac{f(v_i)^{-\alpha}}{\sum_{k=1}^d f(v_k)^{-\alpha}}\right)  + \sum_{i=1}^d u_i \left(-v_i+\frac{f(u_i)^{-\alpha}}{\sum_{k=1}^d f(u_k)^{-\alpha}}\right)          \nonumber\\
             &=&-2\sum_{i=1}^d u_i v_i+\frac{\sum_{i=1}^d u_if(u_i)^{-\alpha}}{\sum_{k=1}^d f(u_k)^{-\alpha}}+\frac{\sum_{i=1}^d v_if(v_i)^{-\alpha}}{\sum_{k=1}^d f(v_k)^{-\alpha}}.
\end{eqnarray}
Notice that the right hand side of (\ref{derivativeformula}) depends on $t$ only through dependence on $u_i(t)$ and $v_i(t)$. We have the following lemma about (\ref{derivativeformula}), which verifies that $L(u,v)$ is a Lyapunov function for a large subset of the domain $D$ according to Definition \ref{definition Lyapunov function}.

\begin{lemma} \label{negative derivative}
Let $D^\delta=\left\{(u,v)\in D:  L(u,v)\ge 3\delta \right\}$. For any fixed $d\ge 3\in \mathbb{N}$, there exists some $\alpha(d)$
independent of $\delta$, s.t. when $\alpha\ge \alpha(d)$
\begin{equation}  \label{derivative}
\left.\frac{d}{dt}(L(\Phi_t(x))) \right|_{(u,v)}\le 0, \quad \forall (u,v)\in D^\delta,
\end{equation}
with equality if and only if $(u,v)=(U,U)$.
\end{lemma}

To prove this lemma, we need several other lemmas. Recall that $V=\{1, \ldots, d\}$.

\begin{lemma}  \label{lemma4}
When $\alpha>d-2$, $U$ (uniform distribution) is a local minimum in $\Delta$ of the  following function
\[g(u_1,\ldots,u_d):=2\min_{i\in V} u_i -d\left(\min_{i\in V} u_i\right)^2 -  \frac{\sum_{i=1}^d u_i^{-\alpha}}{\sum_{i=1}^d
u_i^{-(\alpha+1)}}.
\]
In particular, $g\left(\frac{1}{d},\ldots,\frac{1}{d}\right)=0$.
\end{lemma}

\begin{proof}
Define a function on $(\mathbb{R}^d)^+$:
\[G(w_1,\ldots,w_d)=\frac{2}{\sum_{i=1}^d w_i}\cdot \min_{i\in V} w_i -\frac{d}{(\sum_{i=1}^d w_i)^2}\cdot (\min_{i\in V} w_i)^2 - \frac{1}{\sum_{i=1}^d w_i}\cdot \frac{\sum_{i=1}^d w_i^{-\alpha}}{\sum_{i=1}^d w_i^{-(\alpha+1)}}.
\]
Observe that $G(w_1,\ldots,w_d)$ is a homogeneous function, and that it has the same value as $g(u_1,\ldots,u_d)$ whenever
\[u_i=\frac{w_i}{\sum_{j=1}^d w_j}, \, \forall i \in V.
\]
Let  $W=(w, \ldots, w)$ ($w>0$)  (we refer to $W$ as the diagonal). So to prove the lemma, it suffices to prove that
$W$ is a local minimum of $G(w_1,\ldots,w_d)$.

Without loss of generality, we can assume $w_d=\min_{i\in V} w_i$, then
\begin{equation}
G(w_1,\ldots,w_d)=\frac{2}{\sum_{i=1}^d w_i}\cdot w_d -\frac{d}{(\sum_{i=1}^d w_i)^2}\cdot w_d^2 - \frac{1}{\sum_{i=1}^d w_i}\cdot \frac{\sum_{i=1}^d w_i^{-\alpha}}{\sum_{i=1}^d w_i^{-(\alpha+1)}}.
\end{equation}
By direct calculation, we can check that $G(w_1,\ldots,w_d)$ has zero gradient at $W$, i.e. $\left. \nabla G\right|_{W}=0$.
Then we can calculate the Hessian matrix of $G(w_1,\ldots,w_d)$  at $W$
\[
\left. H\right|_{W} = \frac{2}{d^2w^2}
             \cdot \begin{pmatrix}
-\frac{\alpha+2}{d}+\alpha+1      & \ldots           &-\frac{\alpha+2}{d} & 1-\frac{\alpha+2}{d} \\
\vdots  &       \ddots & \vdots  &\vdots   \\
-\frac{\alpha+2}{d} & \ldots &  -\frac{\alpha+2}{d}+\alpha+1  & 1-\frac{\alpha+2}{d} \\
1-\frac{\alpha+2}{d}   & \ldots & 1-\frac{\alpha+2}{d} & -\frac{\alpha+2}{d}+\alpha+1+2-d
\end{pmatrix}.
\]
Set
\[P=\begin{pmatrix}
-\frac{\alpha+2}{d}+\alpha+1      & \ldots           &-\frac{\alpha+2}{d} & 1-\frac{\alpha+2}{d} \\
\vdots  &       \ddots & \vdots  &\vdots   \\
-\frac{\alpha+2}{d} & \ldots &  -\frac{\alpha+2}{d}+\alpha+1  & 1-\frac{\alpha+2}{d} \\
1-\frac{\alpha+2}{d}   & \ldots & 1-\frac{\alpha+2}{d} & -\frac{\alpha+2}{d}+\alpha+1+2-d
\end{pmatrix},
\]
and
\[Q=\begin{pmatrix}
-\frac{\alpha+2}{d}   & -\frac{\alpha+2}{d}    & \ldots           &-\frac{\alpha+2}{d} & 1-\frac{\alpha+2}{d} \\
-\frac{\alpha+2}{d}    &  -\frac{\alpha+2}{d}  &-\frac{\alpha+2}{d} &\ldots & 1-\frac{\alpha+2}{d} \\
\vdots  &       \vdots &  \vdots & \vdots  &\vdots   \\
-\frac{\alpha+2}{d} & \ldots &-\frac{\alpha+2}{d} &  -\frac{\alpha+2}{d}  & 1-\frac{\alpha+2}{d} \\
1-\frac{\alpha+2}{d}  & \ldots & \ldots & 1-\frac{\alpha+2}{d} & -\frac{\alpha+2}{d}+2-d
\end{pmatrix}.
\]
Notice that $\left. H\right|_{W}$, $P$ and $Q$ satisfy
\[\left.H\right|_{W} = \frac{2}{d^2w^2}P,\quad P=(\alpha+1) I+Q,
\]
where $I$ is the identity matrix.
By direct calculation, we can get all the eigenvalues of matrix $Q$
\[\lambda^Q_1=\ldots=\lambda^Q_{d-2}=0,\lambda^Q_{d-1}=-(\alpha+1),\lambda^Q_d=-\left(d-1\right),
\]
and then get all the eigenvalues of  $P$ by shifting  $Q$'s eigenvalues by $\alpha+1$
\[\lambda^P_1=\ldots=\lambda^P_{d-2}=\alpha+1,\lambda^P_{d-1}=0,\lambda^P_d=\alpha+2-d.
\]
It is easy to see that when  $\alpha>d-2$,  one of $P$'s eigenvalues is zero and all the others are strictly positive.
It is also easy to check that the sum of each row of $P$ is zero, which means
\[P\begin{pmatrix}
1, \ldots, 1
\end{pmatrix}^T=0\cdot\begin{pmatrix}
1, \ldots, 1
\end{pmatrix}^T.
\]
That is, the diagonal is an eigenvector associated with $P$'s  zero eigenvalue and then $H$'s  zero eigenvalue.
Then we can conclude that $G(w_1,\ldots,w_d)$ is convex along all the other directions except the diagonal,
and hence the diagonal is its local minimum.
\end{proof}

Keeping the notations of Lemma \ref{lemma4}, we have the following lemma.
\begin{lemma} \label{lemma2}
For any fixed positive integer $d\ge 3$, there exists some $\alpha_0(d)$, such that when $\alpha>\alpha_0(d)$,
$U$ (uniform distribution) is a global minimum  of $g(u_1,\ldots,u_d)$ in $\Delta$.
\end{lemma}

\begin{proof}
It is equivalent to show that for any $u=(u_1,\ldots,u_d)\in \overset{\circ}{\Delta}$ the following holds
\begin{equation}  \label{inequality2}
2\min_{i\in V} u_i -d\left(\min_{i\in V} u_i\right)^2 \ge  \frac{\sum_{i=1}^d u_i^{-\alpha}}{\sum_{i=1}^d u_i^{-(\alpha+1)}},
\end{equation}
with equality if and only if $u$ is the uniform distribution, i.e. $u=U$.

We will divide the proof of (\ref{inequality2}) into two cases:
\begin{enumerate}[(1)]
\item $u$ is in a neighborhood of uniform distribution;
\item $u$ is bounded away from the uniform distribution. Equivalently, there exists some $0<\kappa<1$ s.t. $\min_{i\in V} u_i< \kappa/d$.
\end{enumerate}

Case (1) directly follows from Lemma \ref{lemma4}. \\

To prove case (2), first we try to use the minimum coordinates of $u$ to bound the right hand side of (\ref{inequality2}) from above.
More precisely, for fixed $d$ and $\alpha$, we will show that for any $u\in \overset{\circ}{\Delta}$ the following holds
\begin{equation}  \label{inequality3}
\frac{\sum_{i=1}^d u_i^{-\alpha}}{\sum_{i=1}^d u_i^{-(\alpha+1)}} \le d^{1/(\alpha+1)}\min_{i\in V} u_i.
\end{equation}

Without loss of generality, we can assume $u_d=\min_{i\in V} u_i$, and then if we set $a_i=\min_{i\in V} u_i/u_i=u_d/u_i\in (0,1]$,
(\ref{inequality3})  is equivalent to  the following inequality with $a_i\in (0,1]\, (i=1,\ldots,d-1)$
\begin{equation} \label{function1}
\frac{1+\sum_{i=1}^{d-1} a_i^{\alpha}}{1+\sum_{i=1}^{d-1} a_i^{1+\alpha}} \le d^{1/(\alpha+1)}.
\end{equation}
To prove (\ref{function1}), observe that by H{\"o}lder's inequality,
\[\left(\frac{1+\sum_{i=1}^{d-1} a_i^{\alpha}}{d}\right)^{1/\alpha}\le\left(\frac{1+\sum_{i=1}^{d-1} a_i^{\alpha+1}}{d}\right)^{1/(\alpha+1)},
\]
i.e.
\[1+\sum_{i=1}^{d-1} a_i^{\alpha}\le d^{1/(\alpha+1)}\left(1+\sum_{i=1}^{d-1} a_i^{\alpha+1}\right)^{\alpha/(\alpha+1)}.
\]
Then
\begin{eqnarray}
\frac{1+\sum_{i=1}^{d-1} a_i^{\alpha}}{1+\sum_{i=1}^{d-1} a_i^{1+\alpha}} &\le& \frac{d^{1/(\alpha+1)}\left(1+\sum_{i=1}^{d-1} a_i^{\alpha+1}\right)^{\alpha/(\alpha+1)}}{1+\sum_{i=1}^{d-1} a_i^{1+\alpha}}  \nonumber \\
                                                              &=& \frac{d^{1/(\alpha+1)}}{\left(1+\sum_{i=1}^{d-1} a_i^{1+\alpha}\right)^{1/(\alpha+1)}}< d^{1/(\alpha+1)} \nonumber,
\end{eqnarray}
proving (\ref{inequality3}). Notice that when $\alpha\ge \log d/\log(2-\kappa)-1$, for any $0< u_d< \kappa/d$ the following inequality holds
\begin{equation} \label{inequality4}
d^{1/(\alpha+1)}\cdot u_d< 2u_d-du_d^2, \quad i.e. \quad d^{1/(\alpha+1)} < 2-du_d.
\end{equation}
Then (\ref{inequality3}) and (\ref{inequality4}) together imply that when $\alpha\ge \log d/\log(2-\kappa)-1$, for any $u$ satisfying $\min_{i\in V} u_i< \kappa/d$
\begin{equation}  \label{faruniform}
2\min_{i\in V} u_i -d\left(\min_{i\in V} u_i\right)^2 >  \frac{\sum_{i=1}^d u_i^{-\alpha}}{\sum_{i=1}^d u_i^{-(\alpha+1)}}.
\end{equation}

Finally, we need to glue the results in the above two cases together. From case (1), for fixed  $d$ and $\alpha>d-2$, there exists a neighborhood
of the uniform distribution $\mathcal{N}(U, \eps_\alpha)$ s.t. for any $u\in \mathcal{N}(U, \eps_\alpha)$, (\ref{inequality2}) holds. For any $u\neq U$
in $\overset{\circ}{\Delta}$, since $\sum_{i=1}^d u_i^{-\alpha}/\sum_{i=1}^d u_i^{-(\alpha+1)}$ is a decreasing
function\footnote{This can be proved by directly checking the derivative.}  in $\alpha$, $g(u_1,\ldots,u_d)$ in Lemma \ref{lemma4} is
an increasing function in $\alpha$. This allows us to take some common neighborhood $\mathcal{N}(U, \eps_d)=\bigcap_{\alpha>d-1}\mathcal{N}(U, \eps_\alpha)$ just depending on $d$ such that (\ref{inequality2}) holds. Take some $\kappa=\kappa(d)<1$ such that
\[\left\{u\in \overset{\circ}{\Delta}:\min_{i\in V} u_i<\frac{\kappa}{d}\right\}\cup \mathcal{N}(U, \eps_d) =\overset{\circ}{\Delta}.
\]
Then set
\[\alpha_0(d)=\max\left\{d-1,\ \frac{\log d}{\log(2-\kappa)}-1 \right\}.
\]
When $\alpha>\alpha_0(d)$, the above two cases imply (\ref{inequality2}).
\end{proof}

\begin{lemma} \label{proposition1}
For any fixed positive integer $d\ge 3$, there exists some $\alpha_0(d)$, such that when $\alpha>\alpha_0(d)$, for any $(u,v)=(u_1,\ldots,u_d,v_1,\ldots,v_d)\in \overset{\circ}{D}$ the following holds
\begin{equation} \label{inequality1}
2\sum_{i=1}^d u_iv_i\ge \frac{\sum_{i=1}^d u_i^{-\alpha}}{\sum_{i=1}^d u_i^{-(\alpha+1)}}+\frac{\sum_{i=1}^d v_i^{-\alpha}}{\sum_{i=1}^d v_i^{-(\alpha+1)}},
\end{equation}
with equality if and only if $u_i=v_i=\frac{1}{d}$ for any $i$, i.e. $(u,v)=(U,U)$ .
\end{lemma}

\begin{proof}
First for any fixed $u,v\in \overset{\circ}{\Delta}$, we will bound the left hand side of (\ref{inequality1}) from below by the minimum coordinates of $u$ and $v$. More precisely, if we construct two $d-$dimensional vectors $u^\prime$ and $v^\prime$ by the minimum coordinates of $u$ and $v$
\[u^\prime=\left(\min_{i\in V}u_i, \ldots,\min_{i\in V}u_i ,1-(d-1)\min_{i\in V}u_i \right)
\]
and
\[v^\prime=\left(1-(d-1)\min_{i\in V}v_i, \min_{i\in V}v_i, \ldots,\min_{i\in V}v_i\right),
\]
we will show
\begin{equation}  \label{minimumbound}
\sum_{i=1}^d u_iv_i \ge \sum_{i=1}^d u^\prime_iv^\prime_i=\min_{i\in V}u_i +\min_{i\in V} v_i-d\cdot \min_{i\in V}u_i \min_{i\in V} v_i.
\end{equation}
By the Rearrangement inequality, it suffices to show (\ref{minimumbound}) for any $(u,v)\in \overset{\circ}{D} $ satisfying $u_1\le\ldots\le u_d$ and $v_1\ge \ldots \ge v_d$. For such $u$ and $v$, we have
\begin{eqnarray}
\sum_{i=1}^d u_i v_i &=& u_1 v_1+ \sum_{i=2}^d u_i v_d +\sum_{i=2}^d u_i (v_i-v_d) \nonumber \\
                     &\ge& u_1 v_1+ \sum_{i=2}^d u_i v_d +\sum_{i=2}^d u_1 (v_i-v_d) \nonumber \\
                     &=& u_1 (1-(d-1)v_d)+ \sum_{i=2}^d u_i v_d  \nonumber \\
                     &=& \sum_{i=1}^d u_iv^\prime_i
                     \ge  \sum_{i=1}^d u^\prime_iv^\prime_i \nonumber,
\end{eqnarray}
where the last step is obtained by repeating the same argument as in the previous steps. Then we have proved (\ref{minimumbound}).

By Lemma \ref{lemma2}, we know that there exists some $\alpha_0(d)$, such that when $\alpha>\alpha_0(d)$, for any $u,v\in \overset{\circ}{\Delta}$ the following inequalities hold
\begin{equation}
2\min_{i\in V}u_i -d\left(\min_{i\in V}u_i\right)^2 \ge  \frac{\sum_{i=1}^d u_i^{-\alpha}}{\sum_{i=1}^d u_i^{-(\alpha+1)}}  \nonumber
\end{equation}
and
\begin{equation}
2\min_{i\in V}v_i -d\left(\min_{i\in V}v_i\right)^2 \ge  \frac{\sum_{i=1}^d v_i^{-\alpha}}{\sum_{i=1}^d v_i^{-(\alpha+1)}}.  \nonumber
\end{equation}
Then by the following
\begin{eqnarray}
& & 2\left(\min_{i\in V}u_i +\min_{i\in V}v_i-d\min_{i\in V}u_i \min_{i\in V}v_i\right) \nonumber \\
&\ge& \left(2\min_{i\in V}u_i -d\left(\min_{i\in V}u_i\right)^2 \right)   +\left(2\min_{i\in V}v_i -d\left(\min_{i\in V}v_i\right)^2\right), \nonumber \end{eqnarray}
we have proved (\ref{inequality1}).

It is also easy to check the equality holds if and only if both $u$ and $v$ are uniform distributions, i.e. $u=v=U$.
\end{proof}

\begin{proofof}{Lemma \ref{negative derivative}}
Recall that $f(x)=\delta 1_{x\le \delta}+x 1_{x>\delta}$. The proof of this lemma is divided into three cases:
\begin{enumerate}[(1)]
\item $\min_{i\in V} f(u_i)>\delta$ and $\min_{i\in V} f(v_i)>\delta$;
\item $\min_{i\in V} f(u_i)=\delta$ and $\min_{i\in V} f(v_i)>2 \delta$ (or the symmetric case);
\item $\min_{i\in V} f(u_i)=\delta$ and $\min_{i\in V} f(v_i)\le 2\delta$ (or the symmetric case).
\end{enumerate}

To prove case (1), observe that $\min_{i\in V} f(u_i)>\delta$ and $\min_{i\in V} f(v_i)>\delta$ imply $f(u_i)=u_i>\delta$ and $f(v_i)=v_i>\delta$
for any $i \in V$. Hence, to prove (\ref{derivative}), it is equivalent to prove (\ref{inequality1}). Then by Lemma \ref{proposition1}, when $\alpha >\alpha_0(d)+1$, $\left. \frac{d}{dt}(L(\Phi_t(x)))  \right|_{(u,v)}\le0$, with equality if and only if $(u,v)=(U,U)$. \\

Let us prove case (2).  From (\ref{derivativeformula}), it follows that
\begin{eqnarray} \label{DBM}
\left.\frac{d}{dt}(L(\Phi_t(x))) \right|_{(u,v)}  &=&-2\sum_{i=1}^d u_i v_i+\frac{\sum_{i=1}^d u_if(u_i)^{-\alpha}}{\sum_{k=1}^d f(u_k)^{-\alpha}}+\frac{\sum_{i=1}^d v_if(v_i)^{-\alpha}}{\sum_{k=1}^d f(v_k)^{-\alpha}} \nonumber \\
                    &\le&-2\sum_{i=1}^d u_i v_i+\frac{\sum_{i=1}^d f(u_i)f(u_i)^{-\alpha}}{\sum_{k=1}^d f(u_k)^{-\alpha}}+\frac{\sum_{i=1}^d f(v_i)f(v_i)^{-\alpha}}{\sum_{k=1}^d f(v_k)^{-\alpha}} \nonumber \\
                    &\le& -2\sum_{i=1}^d u_i v_i+d^{1/\alpha}\left(\min_{i\in V} f(u_i)+\min_{i\in V} f(v_i)\right),
\end{eqnarray}
where the last step is by (\ref{inequality3}), which actually holds for any collection of positive numbers.
Since
\[\sum_{i=1}^d u_i v_i\ge \max\left\{\min_{i\in V} u_i,\min_{j\in V} v_j\right\},
\]
it follows from the conditions in case (2) that
\[\sum_{i=1}^d u_i v_i\ge \max\left\{\min_{i\in V} u_i,\min_{j\in V} v_j\right\}\ge \min_{j\in V} v_j=\min_{i\in V} f(v_i).
\]
Hence,
\begin{eqnarray}
\left.\frac{d}{dt}(L(\Phi_t(x))) \right|_{(u,v)}  &\le& -2\sum_{i=1}^d u_i v_i+d^{1/\alpha}\left(\min_{i\in V} f(u_i)+\min_{i\in V} f(v_i)\right) \nonumber  \\
                    &\le& -2\min_{i\in V} f(v_i)+d^{1/\alpha}\left(\delta+\min_{i\in V} f(v_i)\right)  \nonumber \\
                    &=& -\left(2-d^{1/\alpha}\right)\min_{i\in V} f(v_i)+d^{1/\alpha}\delta \nonumber \\
                    &\le& -\left(2-d^{1/\alpha}\right)2\delta+d^{1/\alpha}\delta= -\left(4-3d^{1/\alpha}\right)\delta. \nonumber
\end{eqnarray}
So one can choose $\alpha>\log{d}/\log{\frac{4}{3}}$ such that  $4-3d^{1/\alpha}>0$, then obtaining $\left. \frac{d}{dt}(L(\Phi_t(x))) \right|_{(u,v)}<0$. \\

To prove case (3), one can choose $\alpha>\log_{2}{d}$ such that $\frac{3}{2}d^{1/\alpha}< 3$. Then by the definition of $D^\delta$, one has
\begin{equation}  \label{lyapunov larger than 3delta}
L(u,v)\ge 3\delta >\frac{3}{2}d^{1/\alpha} \delta, \quad \forall (u,v)\in D^\delta.
\end{equation}
(\ref{DBM})  and (\ref{lyapunov larger than 3delta}) imply
\begin{equation}
\left.\frac{d}{dt}(L(\Phi_t(x))) \right|_{(u,v)}  <-2\sum_{i=1}^d u_i v_i+3 d^{1/\alpha} \delta<0.
\end{equation}

To sum up, taking
\[\alpha(d)=\max\left\{\alpha_0(d)+1,\ \log{d}/\log{\frac{4}{3}}, \ \log_{2}{d}  \right\},
\]
we establish the lemma.
\end{proofof}

\subsubsection{The main lemma}   \label{the main lemma}

\textbf{} Now it comes to our main lemma to characterize the chain recurrent set for our specific semiflow $\Phi$.

\begin{lemma} \label{limitrecurrent}
Assume the result in Lemma \ref{negative derivative}. Let
\[S^\delta=\{(u,v)\in D: \sum_{i=1}^d u_iv_i\le 4\delta\}.
\] Then
$\CR{\Phi} \subset S^\delta \cup (U,U)$.
\end{lemma}

\begin{proof}
Set $\zeta_0=3.5\delta$, $\zeta_1=4\delta$ and $\zeta_2=1$. Define
\[M_j=\{(x,y)\in D: L(x,y)\le \zeta_j\}, \quad j=0,1,2.
\]
Note that  $M_1=S^\delta, M_2=D$. By Lemma \ref{negative derivative} and Proposition \ref{proposition10}, $M_j(j=0,1,2)$ are compact invariant sets.
Clearly, the lemma will follow once we prove:
\begin{enumerate}[(a)]
\item $CR_1=\CR{\Phi}\cap M_1$ and $CR_2=\CR{\Phi}\cap (M_2 \setminus M_1)$ are invariants sets;
\item $CR_2 =(U,U)$.
\end{enumerate}

Let's prove (a).

By the invariance of $\CR{\Phi}$ and $M_1$, it is clear that $CR_1$ is invariant.

We will show $CR_2$ is invariant by contradiction. Suppose for some $z\in CR_2$, there exists some $T_0>0$, s.t. $\Phi_{T_0}(z)\in M_1$. Then by Lemma \ref{negative derivative} and compactness of $\overline{M_1 \backslash M_0}$, there exists some $T_1>T_0$, such that
\begin{equation}  \label{smallball}
L(\Phi_{T_1}(z))<\zeta_0,\ i.e. \ \Phi_{T_1}(z) \in M_0.
\end{equation}
Also by Lemma \ref{negative derivative} and compactness of $\overline{M_1 \backslash M_0}$, there exists some $T_2>0$, s.t.
\[\Phi_{T_2}(M_1)\subset M_0.
\]
Let $\rho_0=\dist{M_0}{\overline{D\backslash M_1}}>0$ and $T=\max\{T_1, T_2\}$. By assumption, there are points $z_0=z,z_1,\ldots,z_{k-1},z_k=z\in D$ and real
numbers $t_0,\ldots,t_{k-1}>T$ such that
\begin{align}\label{condition chain-recurrence}
\dist{\Phi_{t_i}(z_i)}{z_{i+1}}<\rho_0, \ \ \ i=0,\ldots,k-1.
\end{align}
Thus $\Phi_{t_0}(z_0)=\Phi_{t_0}(z)\in M_0$ and so, by (\ref{condition chain-recurrence}), $z_1\in M_1$. By induction, we claim that $z_1,z_2,\ldots,z_k\in M_1$. Indeed, if $z_i\in M_1$, $\Phi_{t_i}(z_i)\in M_0$, and then by (\ref{condition chain-recurrence}), $z_{i+1}\in M_1$. In particular, $z_k=z\in M_1$, which contradicts. Hence, $\Phi_t(CR_2)\subset M_2\setminus M_1$ for all $t\ge0$. By invariance of $\CR{\Phi}$, $CR_2$ is invariant.

It remains to prove (b).

Since $(U,U)\in CR_2$, it suffices to show $CR_2\subset (U,U)$. For any $z\in CR_2$, by invariance of $CR_2$ and the non-increasing property of the Lyapunov function $L(\cdot)$ along any trajectory in $M_2\setminus M_1$, it follows that the limit of $L(\Phi_t(z))$ exists. Set
\[L(\Phi_{\infty}(z))=\lim_{t\to +\infty}L(\Phi_t(z)).
\]
Then for any $p\in \omega(z)$ (the omega limit set of $z$), $L(p)=L(\Phi_{\infty}(z))$. Together with invariance of $\omega(z)$, this implies $L(\cdot)$ is constant along trajectories in $\omega(z)$. Therefore, $\omega(z)\subset (U,U)$. Since $\omega(z)$ is nonempty, $\omega(z)=(U,U)$.

Now we will show $z=(U,U)$ by contradiction. Suppose $z\neq (U,U)$, then there exists a neighborhood of $z$, s.t. $L(\Phi_t(z))$ is strictly decreasing in this neighborhood. Since
\[L(\Phi_t(z))\ge L(\omega(z))=L(U,U)=1/d,
\]
there exists some $\epsilon>0$, s.t. $L(z)>1/d+\epsilon$.
Define
\[E=\left\{w \in D\left|L(w)\le \frac{1}{d}+\frac{\epsilon}{3}\right. \right\}, \quad F=\left\{w \in D\left|L(w)\ge \frac{1}{d}+\frac{\epsilon}{2}\right. \right\}.
\]
By by Lemma \ref{negative derivative} and compactness of $F$, there exists some $T$, such that
$\Phi_T(F)\subset E$.
Let $\rho_1=\dist{E}{F}>0$. Then by a similar argument as that in (a), one can get a contradiction.
\end{proof}

\subsection{Non-convergence to unstable equilibrium}  \label{Non-convergence to unstable equilibrium}

By Theorem \ref{theorem benaim} and Lemma \ref{limitrecurrent},
\[\mathbb{P}\left( L(\left \{ z(n) \right\}_{n\ge 0})\subset S^\delta \cup (U,U)  \right)=1.
\]
Since $S^\delta$ and $(U,U)$ are disconnected, we can finish the proof of Theorem \ref{main theorem}
by proving the following lemma.

\begin{lemma} \label{nonuniform}
When $\alpha>1$, for any initial condition $(x(0),y(0))\in D$, $z(n)$ in (\ref{main random process}) satisfies
\[ \mathbb{P}\left(\lim_{n\to \infty}z(n)=(U,U)\right)=0.
\]
\end{lemma}

Lemma \ref{nonuniform} is an application of the following theorem to our case. This theorem is due to Pemantle ~\cite{pemantle1990nonconvergence} which, to our purpose, is stated as

\begin{theorem}{~\cite[Theorem 1]{pemantle1990nonconvergence}}  \label{unstableanalysis}
Define a stochastic process
\[z(n+1)-z(n)=\frac{1}{n+1+d} [F(z(n))+u_n]
\]
with $\mathbb{E}(u_n|\mathcal{F}_n)=0$ and such that $z(n)$ always remains in bounded domain $D$. Let $p$ be any point in $\overset{\circ}{D}$ with $F(p)=0$, let $\mathcal{N}$ be a neighborhood of $p$ and assume that there are constants  $c_1,c_2>0$ for which the following conditions  are satisfied whenever $z(n)\in \mathcal{N}$ and $n$ is sufficiently large:
\begin{enumerate}[(1)]
\item $p$ is a linearly unstable critical point,
\item $\mathbb{E}((u_n\cdot\theta)^+|\mathcal{F}_n)\ge c_1$ for every unit vector $\theta\in TD$ (see the definition of $TD$ in Notation \ref{notation_collection}),
\item $\Vert u_n \Vert \le c_2$,
\end{enumerate}
where $(u_n\cdot\theta)^+=\max\{u_n\cdot\theta,0\}$ is the positive part of $u_n\cdot\theta$. Assume $F$ is smooth enough to apply the stable manifold theorem: at least $C^2$. Then
\[\mathbb{P}\left(\lim_{n\to \infty}z(n)= p\right)=0.
\]
\end{theorem}

The rest of this section is to verify that our $z(n)$ in (\ref{main random process}) falls into the setting of Theorem \ref{unstableanalysis} with $p=(U,U)$. First it is easy to check that $p$  is a critical point of the vector field $F$ in (\ref{definition F, part  1}), i.e. $F(U,U)=0$. After introducing the following definition, we will show $p$ is a linearly unstable critical point.

\begin{definition}
Let $T$ be the linear approximation to some vector field $F$ near a critical point $p$ so that $F(p+w)=T(w)+O(|w|^2)$, then
\begin{enumerate}[(a)]
\item If all the eigenvalues of $T$ have strictly negative real part, $p$ is called an attracting point.
\item If some eigenvalues of $T$ have strictly positive real part, $p$ is called an unstable critical point.
\end{enumerate}
\end{definition}

\begin{lemma} \label{typeoffixedpoint}
\begin{enumerate}
\item When $\alpha>1$, $(U,U)$ is a linearly unstable critical point for $F$ in (\ref{definition F, part  1}).
\item When $\alpha<1$, $(U,U)$ is an attracting  point for $F$ in (\ref{definition F, part  1}).
\end{enumerate}
\end{lemma}

\begin{proof}
By Taylor expansion of $F$ in a neighborhood of $(U,U)$, we have
\[F(p+w)=\left. DF\right|_p\cdot w+O(|w|^2),
\]
where $DF$ is Jacobian matrix and $w$ is some vector in a neighborhood of 0 (2d dimensional vector).
By direct calculation, one can have
\[
\left. DF\right|_p=\left(\begin{array}{cccc|cccc}
-1     &  0     & \ldots & 0        &       -\alpha+\frac{\alpha}{d}  & \frac{\alpha}{d} &  \ldots   & \frac{\alpha}{d} \\
0      & -1     & 0      & \ldots   &       \frac{\alpha}{d}          & -\alpha+\frac{\alpha}{d}     & \frac{\alpha}{d}    & \ldots      \\
\vdots & \vdots & \ddots & \vdots   &       \vdots                    & \vdots & \ddots              & \vdots    \\
0  & \ldots     &    0   & -1       &        \frac{\alpha}{d}         & \ldots & \frac{\alpha}{d}    &   -\alpha+\frac{\alpha}{d}       \\ \hline
-\alpha+\frac{\alpha}{d}  & \frac{\alpha}{d} &  \ldots   & \frac{\alpha}{d}  &   -1   &  0 & \ldots & 0  \\
\frac{\alpha}{d}          & -\alpha+\frac{\alpha}{d}     & \frac{\alpha}{d}    & \ldots            &   0   & -1 & 0 & \ldots   \\
\vdots                    & \vdots & \ddots              & \vdots     &  \vdots & \vdots & \ddots & \vdots  \\
 \frac{\alpha}{d}         & \ldots & \frac{\alpha}{d}    &   -\alpha+\frac{\alpha}{d}  & 0  & \ldots &  0 & -1 \\
\end{array}
\right).
\]
In order to get all the eigenvalues of $\left. DF\right|_p$, one need to solve
\begin{equation} \label{unstableeigenvalue}
\left|\left. DF\right|_p-\lambda I_{2d\times 2d}\right|=0,
\end{equation}
where $|\cdot|$ represents the matrix's determinant and $I_{2d\times 2d}$ is a $2d$ dimensional identity matrix. Notice that  $\left. DF\right|_p$ has the same upper-right and lower-left block matrix. We set this $d\times d$ matrix as
\[B=\begin{pmatrix}
-\alpha+\frac{\alpha}{d}  & \frac{\alpha}{d} &  \ldots   & \frac{\alpha}{d}  \\
\frac{\alpha}{d}          & -\alpha+\frac{\alpha}{d}     & \frac{\alpha}{d}    & \ldots            \\
\vdots                    & \vdots & \ddots              & \vdots     \\
 \frac{\alpha}{d}         & \ldots & \frac{\alpha}{d}    &   -\alpha+\frac{\alpha}{d}
\end{pmatrix}.
\]
Because the sum of $B$'s each row is zero, $|B|=0$. Then one can easily check $\lambda=-1$ is a solution to (\ref{unstableeigenvalue}).  Now assume $\lambda\neq -1$.
By Schur complement, from (\ref{unstableeigenvalue}) we get
\begin{equation} \label{unstableeigenvalue2}
\left|B^2-(\lambda+1)^2I_{d\times d}\right|=0,
\end{equation}
where $I_{d\times d}$ is a $d$ dimensional identity matrix. (\ref{unstableeigenvalue2}) is equivalent to
\begin{equation} \label{unstableeigenvalue3}
\left|B-(\lambda+1)I_{d\times d}\right|\cdot\left|B+(\lambda+1)I_{d\times d}\right|=0.
\end{equation}
Under the assumption $\lambda\neq -1$, we easily get all the solutions to (\ref{unstableeigenvalue3}): $ \lambda=-1\pm \alpha$. Hence all the eigenvalues of $\left. DF\right|_p$ without counting multiplicities are
\[-1, -1\pm \alpha.
\]
So if $\alpha>1$, $\left. DF\right|_p$ has the positive eigenvalue $-1+\alpha$; If $\alpha<1$, all of its eigenvalues are strictly negative. This concludes the proof.
\end{proof}

Clearly, $u_n$ in (\ref{definition u_n}) satisfies condition (3) of Theorem \ref{unstableanalysis}. It remains to check condition (2), which is the following lemma.

\begin{lemma} \label{condition3}
In a small neighborhood of $p=(U,U)$, there exists some constant $c>0$, s.t.
$\mathbb{E}((u_n\cdot\theta)^+|\mathcal{F}_n)\ge c$ for every unit vector $\theta=(\theta_k)_{1\le k\le 2d} \in TD$.
\end{lemma}

\begin{proof}
Conditioning on $X_{n+1}=i\in V, Y_{n+1}=j\in V$, one has
\begin{eqnarray} \label{conditionalerror}
u_n\cdot \theta &=& (1-\pi_i(y(n)))\theta_i-\sum_{m\neq i, m\in V}\pi_m(y(n))\theta_m \nonumber \\
                              & & +(1-\pi_j(x(n)))\theta_{j+d}-\sum_{s\neq j, s\in V}\pi_s(x(n))\theta_{s+d} \nonumber \\
                              &=& \theta_i+\theta_{j+d}-\sum_{m\in V}\pi_m(y(n))\theta_m -\sum_{s\in V}\pi_s(x(n))\theta_{s+d}. \qquad
\end{eqnarray}
Now we will show for any unit vector  $\theta\in TD$, the following holds
\begin{equation} \label{lowerbound}
\max_{1\le k\le 2d} \theta_k \ge \frac{1}{2d(d-1)}.
\end{equation}
Observe that $\theta$ as a unit vector ($\Vert \theta \Vert=1$) always satisfies $\max_{1\le k\le 2d} |\theta_k| \ge \frac{1}{2d}$.
Then if the following inequality doesn't hold
\[\max_{1\le k\le 2d} \theta_k \ge \frac{1}{2d} \left(> \frac{1}{2d(d-1)}\right),
\]
there must be $\min_{1\le k\le 2d} \theta_k \le - 1/(2d)$.
Then it follows from  $\sum_{k=1}^d \theta_k=0$ and $\sum_{k=d+1}^{2d} \theta_k=0$ that there exists some coordinate $1\le k_0 \le 2d$, s.t.
$\theta_{k_0}\ge  1/[2d(d-1)]$, which implies (\ref{lowerbound}).
From (\ref{lowerbound}), without loss of generality, we can assume
\begin{equation} \label{assumptiontheta1}
\theta_1=\max_{1\le k\le 2d} \theta_k \ge \frac{1}{2d(d-1)}.
\end{equation}
Then by $\sum_{k=d+1}^{2d} \theta_k=0$, there also exists some $j_0\in V$, s.t. $\theta_{j_0+d}\ge 0$.

Because $(x(n),y(n)$ lives in a small neighborhood of $(U,U)$, $\pi(x(n))$ and $\pi(y(n))$ also live in a small neighborhood of $(U,U)$, and hence in (\ref{conditionalerror}),
\[\sum_{m\in V}\pi_m(y(n))\theta_m \sim 0, \sum_{s\in V}\pi_s(x(n))\theta_{s+d}\sim 0.
\]
Therefore, by (\ref{conditionalerror}),
\begin{eqnarray}
\mathbb{E}((u_n\cdot\theta)^+|\mathcal{F}_n)&\ge& \mathbb{P}(X_{n+1}=1, Y_{n+1}=j_0|\mathcal{F}_n) \theta_1  \nonumber \\
                                                         &=& \mathbb{P}(X_{n+1}=1|\mathcal{F}_n) \mathbb{P}( Y_{n+1}=j_0|\mathcal{F}_n) \theta_1.  \nonumber
\end{eqnarray}
Then by the same fact that $\pi(x(n))$ and $\pi(y(n))$ live in a small neighborhood of $(U,U)$, both $\mathbb{P}(X_{n+1}=1|\mathcal{F}_n)$ and $\mathbb{P}( Y_{n+1}=j_0|\mathcal{F}_n)$
are close to $1/d$. Then together with (\ref{assumptiontheta1}), it follows that $\mathbb{E}((u_n\cdot\theta)^+|\mathcal{F}_n)$ is uniformly bounded from below by some positive constant. This completes the proof.
\end{proof}

Now we can apply Theorem \ref{unstableanalysis}, obtaining Lemma \ref{nonuniform}.

\section{Further problems}   \label{further problems}

This paper is part of a project to answer the following

\begin{problem} \label{main_problem}
For the model defined in (\ref{Xmove}) and (\ref{Ymove}), what is the random walks' behavior when choosing different graphs and weights?
\end{problem}
One concrete version of the above problem can be the nearest-neighbor interacting random walks on $\Z^d$ with the weight sequence $w_k=k^{-\alpha}$ ($\alpha>0$). Problem \ref{main_problem} remains widely open in general.

Now we make a conjecture regarding the model in (\ref{completemeasureXmove}) and (\ref{completemeasureYmove}) we have studied.
\begin{conjecture}
For any positive integer $d\ge 3$, $\alpha_c=1$ is critical, i.e. for any given initial condition $(x(0),y(0))\in D$,
\begin{enumerate}
\item when $\alpha>1$, there exists some constant $c=c(\alpha,d)$ (not large) depending on $\alpha$ and $d$, such that the following holds
\[\mathbb{P} \left\{ \exists n_0, \bigcap_{n\ge n_0}\left\{ \sum_{i=1}^d x_i(n)y_i(n)\le c\delta \right\} \right\}=1.
\]
\item when $0<\alpha<1$, the following holds
\[\mathbb{P} \left\{ \lim_{n\to \infty}z(n)=(U,U) \right\}=1.
\]
\end{enumerate}
\end{conjecture}
Another problem of interest is
\begin{problem} \label{originalquestion}
When $\delta=0$ in (\ref{completemeasureXmove}) and (\ref{completemeasureYmove}), how to derive the similar result as Theorem \ref{main theorem}?
\end{problem}

\section{Acknowledgements}

This work was supported by the ISF. The author thanks Itai Benjamini for his very stimulating  questions and Gady Kozma for reading the draft and giving many comments. The author is indebted to Ofer Zeitouni for his patient guidance and numerous insights on this work. The author is also thankful to Omri Sarig and Yuri Lima for all the discussions on dynamical systems.

\bibliography{complete_graph_4_bib}

\end{document}